\documentclass[a4paper,12pt,centertags]{amsart}
\usepackage[latin1]{inputenc}
\usepackage[english]{babel}
\usepackage{cite}
\usepackage{verbatim}
\usepackage{color}
\usepackage{txfonts}

\usepackage{graphicx,picture}
\usepackage{amsmath,amsthm}
\usepackage{hyperref}

\usepackage[top=3.2cm, bottom=3.2cm, left=2.8cm,right=2.8cm]{geometry}

\newcommand{\diesis}{^\#}
\newtheorem{theo}{Theorem}[section]
\newtheorem{lemma}{Lemma}[section]
\newtheorem{prop}{Proposition}[section]

\theoremstyle{definition}
\newtheorem{definiz}{Definition}[section]
\newtheorem{rem}{Remark}[section]
\newtheorem{ex}{Example}[section]
\numberwithin{equation}{section}

\newcommand{\R}{\mathbb R}
\newcommand{\de}{\partial}
\newcommand{\eps}{\varepsilon}

\DeclareMathOperator{\Sk}{S}

\begin{document}
\title[Upper bounds for the eigenvalue of $\Sk_k$]{Upper bounds
  for the eigenvalues of Hessian equations}
\author[F. Della Pietra, N. Gavitone]{
  Francesco Della Pietra and Nunzia Gavitone
}
\address{Francesco Della Pietra \\
Universit\`a degli studi del Molise \\
Dipartimento di Bioscienze e Territorio \\
Via Duca degli Abruzzi \\
86039 Termoli (CB), Italia.
}
\email{francesco.dellapietra@unimol.it}

 \address{
Nunzia Gavitone \\
Universit\`a degli studi di Napoli ``Federico II''\\
Dipartimento di Matematica e Applicazioni ``R. Caccioppoli''\\
80126 Napoli, Italia.
}
\email{nunzia.gavitone@unina.it}
\date{\today}
\maketitle
\begin{abstract}
  In this paper we prove some upper bounds for the Dirichlet
  eigenvalues of a class of fully nonlinear elliptic equations, namely
  the Hessian equations.
\end{abstract}
\section{Introduction}
In this paper we deal with the eigenvalue problem of the $k$-Hessian
operator, namely
\begin{equation}
  \label{autskintro}
  \left\{
    \begin{array}{ll}
      \Sk_k(D^2u)=\lambda (-u)^k &\text{in } \Omega,\\
      u=0 &\text{on } \de\Omega,
    \end{array}
  \right.
\end{equation}
where $1\le k\le n$, and $\Omega$ is a bounded, strictly convex, open set
of $\R^n$, $n\ge 2$, with $C^2$ boundary. Here $\Sk_k(D^2u)$ is the $k$-th
elementary symmetric function of the eigenvalues of the
Hessian matrix of $u\in C^2(\Omega)$ (see Section 2 for the
precise definitions). Notice that for $k=1$, $S_1(D^2u)$ reduces to
the Laplacian operator $\Delta u$, while for $k=n$, $S_n(D^2u)$ is the
Monge-Amp\`ere operator $\det D^2u$.

Our aim is to generalize some well-known estimates
involving the first eigenvalue $\lambda_1(\Omega)$ of the
Dirichlet-Laplacian in $\Omega$. In this case, the Faber-Krahn
inequality states that $\lambda_1(\Omega)$ attains its minimum at the
ball $\Omega\diesis$ with the same Lebesgue measure of $\Omega$,
that is 
\begin{equation}
  \label{fkintro}
\lambda_1(\Omega\diesis)\le \lambda_1(\Omega).
\end{equation}
Hence, a natural question which arises from \eqref{fkintro} is to give
an upper bound of $\lambda_1(\Omega)$. For example, in \cite{mak} it
is proved that for a convex plane domain of area $|\Omega|$ and perimeter
$P(\Omega)$, 
\begin{equation}
  \label{makin}
  \lambda_1(\Omega) \le 3 \frac{P^2(\Omega)}{|\Omega|^2}.
\end{equation}
The constant $c=3$ is not sharp, and P\'olya in \cite{po61} has shown that
it can be replaced by $\pi^2/{4}$. Moreover, the inequality holds for any
simply connected bounded open set $\Omega$ of $\R^2$.

Another classical result, due to Payne and Weinberger (see
\cite{pw61}), allows to obtain an upper bound of $\lambda_1(\Omega)$
in terms of $\lambda_1(\Omega\diesis)$ and the isoperimetric
deficit. More precisely, if $\Omega$ is a simply connected, bounded
open set of $\R^2$ with smooth boundary, then
\begin{equation}
  \label{pw}
  \lambda_1(\Omega) \le \lambda_1(\Omega\diesis) \left[1+C
    \left(\frac{P^2(\Omega)}{4\pi |\Omega|}-1\right) \right],
\end{equation}
where $C$ is a universal sharp constant, which can be explicitly
determined. Hence, together with the Faber-Krahn inequality it is
possible to obtain a stability estimate for $\lambda_1(\Omega)$, that
is 
\begin{equation}
  \label{stab1}
  0\le  \frac{ \lambda_1(\Omega) -
    \lambda_1(\Omega\diesis)}{\lambda_1(\Omega\diesis)} \le C
    \left(\frac{P^2(\Omega)}{4\pi |\Omega|}-1\right).
  \end{equation}
  Recently, an estimate of this kind, which involves an isoperimetric
  deficit of $\Omega$, has been obtained in the paper \cite{bnt10}
  for a larger class of operators in any dimension. In particular, the
  authors prove that, if $\Omega$ is a bounded convex open set of $\R^n$,
  then
  \begin{equation}
  \label{stab2}
  \frac{ \lambda_{1,p}(\Omega) -
    \lambda_{1,p}(\Omega^\star)}{\lambda_{1,p}(\Omega)} \le C(n,p,\Omega)
    \left(1-\frac{n^{\frac{n}{n-1}}\omega_n^{\frac{1}{n-1}}|\Omega|}
      {P(\Omega)^{\frac{n}{n-1}}}\right),
  \end{equation}
  where $\lambda_{1,p}(\Omega)$ is the first Dirichlet eigenvalue for
  the $p$-Laplace operator, $\Omega^\star$ is the ball centered at
  the origin with the same perimeter of $\Omega$. As matter of fact,
  being $\lambda_{1,p}(\Omega^\star)\le \lambda_{1,p}(\Omega\diesis)$,
 together with the Faber-Krahn inequality of the $p$-Laplacian, we
 have that the left-hand side of \eqref{stab2} is nonnegative.

 The main idea in order to prove the quoted estimates is to make use
 of a particular class of test functions, depending on the distance to
 the boundary, introduced in \cite{mak}, \cite{po61} and nowadays
 known as web functions (see for example \cite{cgweb}). 


 
 The aim of the paper is to prove estimates for the eigenvalue of
 \eqref{autskintro} in the same spirit of \eqref{makin} and
 \eqref{stab2}, when $\Omega$ is a bounded, strictly convex, open set
 with $C^2$ boundary. In particular, we show that if $1\le k\le n$, a
 Makai-type estimate holds, namely 
 \begin{equation}
   \label{makintro}
   \lambda_k(\Omega) \le  \frac{n(k+2)}{n-k +1}
   \frac{P(\Omega)^{k+1}}{|\Omega|^{k+2}} W_{k-1}(\Omega).
 \end{equation}
Here $\lambda_k(\Omega)$ denotes the eigenvalue of $S_k$ in $\Omega$,
and $W_{k-1}(\Omega)$ is the $(k-1)$-th quermass\-integral of
$\Omega$ (see Section 2 for the precise references and
definitions). 
In the Laplacian case, with $k=1$ and $n=2$, we recover
exactly \eqref{makin}. In the Monge-Amp\`ere case, it is worth to
compare \eqref{makintro} with the upper bound obtained in
\cite{bntpoin} (see Remark \ref{cri} and Example \ref{exellipse})

Regarding to the stability estimates, our results read as follows.
If $k=n$, we will prove that
\begin{equation}
   \label{tesiintro}
  \frac{\lambda_n(\Omega) -
    \lambda_n(\Omega^*_{n-1})}{\lambda_n(\Omega)} \le C_\Omega
  \left(|\Omega^*_{n-1}|-|\Omega| \right).
\end{equation}

 Here $\Omega^*_i$, $i=0,\ldots n-1$ denotes the ball centered at the
 origin with the same $i$-th quermass\-integral of $\Omega$. Hence, in
 conjunction with the Faber-Krahn inequality for the 
 Monge-Amp\`ere operator (see \cite{bt07} and \cite{ga09}), 
 the left-hand side of \eqref{tesiintro} is nonnegative and we have a
 stability estimate of $\lambda_n(\Omega)$. 
 
In the case $1\le k\le n-1$, we will obtain that
\begin{equation}
  \label{stab3}
    \frac{\lambda_k(\Omega) -
      \lambda_k(\Omega^*_{k})}{\lambda_k(\Omega)} \le C_\Omega
    \left(|\Omega^*_k|-|\Omega| \right).
  \end{equation}
Again, under suitable assumptions on $\Omega$, the above inequality,
in conjunction with the Faber-Krahn inequality for $\Sk_k$
(see \cite{ga09} and Section 2.2), gives a quantitative estimate.

The paper is organized as follows. In Section 2, we recall some basic
definitions of convex analysis and the main properties of the
eigenvalues of $\Sk_k$. Then, in Section 3 we prove some preliminary
results necessary to prove the main results. In particular, we cannot
apply directly the method of web functions, since they are not
sufficiently regular in order to be used as test functions in
\eqref{autskintro}. Then, we construct a suitable smooth
approximating sequence of the distance function.  
Finally, in Section 4 we state precisely the main results and give the
proofs. 

\section{Notation and preliminaries}
Throughout the paper, we will denote with $\Omega$ a set of
$\R^n$, $n\ge 2$ such that
\begin{equation}
  \label{ipomega}
    \Omega\text{ is a bounded, strictly convex, open set with }C^2
    \text{ boundary}.
\end{equation}
Given a function $u \in C^2(\Omega)$, we denote by $\lambda(D^2u)=
(\lambda_1,\lambda_2, \ldots,\lambda_n)$ the vector of the 
eigenvalues of $D^2u$.  The $k$-Hessian operator $\Sk_k(D^2u)$, with
$k=1,2,\ldots,n$, is
\begin{equation}
  \Sk_k(D^2u)=\sum_{i_1<i_2<\cdots<i_k} \lambda_{i_1} \cdot
  \lambda_{i_2} \cdots \lambda_{i_k}.
\end{equation}
Hence $\Sk_k(D^2u)$ is the sum of all $k \times k$ principal minors of
the matrix $D^2u$.

The $k$-Hessian operator can be written also in divergence form,
that is
\begin{equation}
  \label{div}
  \Sk_k(D^2u)=\frac{1}{k}\sum_{i,j=1}^n (\Sk_k^{ij}u_i)_j,
\end{equation}
where $S_k^{ij} = \frac{\partial S_k(D^2u)}{\partial u_{ij}}$ (see
for instance \cite{trudi1}, \cite{trudi2}, \cite{wangeigen}).

Well known examples ok $k$-Hessian operators are $\Sk_1(D^2u)=\Delta
u$, the Laplace operator, and $\Sk_n(D^2u)=\det(D^2u)$, the
Monge-Ampère operator.

It is well-known that $\Sk_1(D^2u)$ is elliptic. This property is not
true in general for $k>1$. As matter of fact, the $k$-Hessian operator
is elliptic when it acts on the class of the so-called
$k$-convex function, defined below.
\begin{definiz}
Let $\Omega$ as in \eqref{ipomega}. A function $u \in C^2(\Omega)$ is
a $k$-convex function (strictly 
$k$-convex) in $\Omega$ if
\begin{equation}
\Sk_j(D^2u)\geq 0 \text{ }(>0) \quad \text{for }j=1,
\ldots, k.
\end{equation}
We denote the class of $k$-convex functions in $\Omega$ such
that $u \in C^2(\Omega)\cap C(\bar{\Omega})$ and  $u=0$ on $\partial
\Omega$ by $\Phi^2_k(\Omega)$. 
\end{definiz}
If we define with $\Gamma_k$ the following convex open cone
\begin{equation*}
  \Gamma_k=\{ \lambda \in \R^n : S_1(\lambda)>0, S_2(\lambda)>0,
  \ldots, S_k(\lambda)>0\},
\end{equation*}
in \cite{ivo} it is proven that $\Gamma_k$ is the cone of
ellipticity of $\Sk_k$. Hence the $k$-Hessian operator is elliptic
with respect to the $k-$convex functions.

If $u$ is $k$-convex, the following Newton inequalities hold:
\begin{equation}
  \label{newton}
  \frac{\Delta u}{n} \ge \ldots \ge \left(
    \binom{n}{k-1}^{-1} {S_{k-1}(D^2 u)} \right)^{\frac {1} {k-1}} \ge
  \left(\binom{n}{k}^{-1}S_{k}(D^2u)\right)^{\frac{1}{k}}.
\end{equation}

By \eqref{newton} it follows that the $k$-convex functions equal to
zero on the boundary of $\Omega$ are negative in $\Omega$.

We go on by recalling some definitions of convex analysis which will
be largely used in next sections. Standard references for this topic
are \cite{bz}, \cite{schn}.

\subsection{Quermassintegrals and Alexsandrov-Fenchel inequalities}

Let $K$ be a convex body, and let $\rho>0$. We denote with $|K|$ the
Lebesgue measure of $K$, with $P(K)$ the perimeter of $K$ and with
$\omega_n$ the measure of the unit ball in $\R^n$.

The well-known Steiner formula for the Minkowski sum is
\[
|K+\rho B_1| =\sum_{i=0}^{n} \binom{n}{i} W_i(K) \rho^i.
\]
The coefficient $W_i(K)$, $i=0,\ldots,n$ is known as the $i$-th
quermassintegral of $K$. Some special cases are $W_0(K)=|K|$,
$nW_1(K)=P(K)$, $W_n(K)=\omega_n$. If $K$ as $C^2$ boundary, with
nonvanishing Gaussian curvature, the quermassintegrals can be related
to the principal curvatures of $\de K$. Indeed, in such a case
\[
W_i(K)=\frac 1 n \int_{\de K} H_{i-1} d \mathcal H^{n-1}, \quad
i={1,\ldots n}.
\]
Here $H_j$ denotes the $j-$th normalized elementary symmetric function
of the principal curvatures $\kappa_1,\ldots,\kappa_{n-1}$ of $\de K$,
that is $H_0=1$ and
\[
H_j= \binom{n-1}{j}^{-1} \sum_{1\le i_1\le \ldots \le i_j\le n-1}
\kappa_{i_1}\cdots \kappa_{i_j},\quad j={1,\ldots,n-1}.
\]

An immediate computation shows that if $B_R$ is a ball of radius $R$,
then
\begin{equation}
  \label{querball}
  W_i(B_R)= \omega_n R^{n-i}, \quad i=0,\ldots,n.
\end{equation}

A Steiner formula holds true also for every quermassintegral, that is
\[
W_p(K+\rho B_1)= \sum_{i=0}^{n-p} \binom{n-p}{i} W_{p+i}(K)\rho^i,\quad
p=0,\ldots,n-1.
\]
This formula immediately gives that
\begin{equation}
  \label{steinercons}
  \lim_{\rho\rightarrow 0^+}\frac{W_p(K+B_\rho)-W_p(K)}{\rho}=
  (n-p)W_{p+1}(K),\quad p=0,\ldots,n-1.  
\end{equation}

The Aleksandrov-Fenchel inequalities state that
\begin{equation}
  \label{afineq}
\left( \frac{W_j(K)}{\omega_n} \right)^{\frac{1}{n-j}} \ge \left(
  \frac{W_i(K)}{\omega_n} \right)^{\frac{1}{n-i}}, \quad 0\le i < j
\le n-1,
\end{equation}
where the inequality is replaced by an equality if and only if $K$ is
a ball.

In what follows, we use the Aleksandrov-Fenchel inequalities for
particular values of $i$ and $j$. If $i=1$, and $j=k-1$, we have that
\begin{equation}\label{af-2}
W_{k-1}(K) \ge \omega_n^{\frac{k}{n-1}} n^{\frac{n-k+1}{n-1}}
P(K)^{\frac{n-k+1}{n-1}}, \quad
3\le k \le n-1.
\end{equation}
When $i=0$ and $j=1$, we have the classical isoperimetric inequality:
\[
P(K) \ge n \omega_n^{\frac 1 n} |K|^{1-\frac 1 n}.
\]
Moreover, if $i=k-1$, and $j=k$, we have
\[
W_k(K) \ge \omega_n^{\frac{1}{n-k+1}} W_{k-1}(K)^{\frac{n-k}{n-k+1}}.
\]

\subsection{Eigenvalue problems for $\Sk_k$}
Let us consider the eigenvalue problem associated to
$k$-Hessian operator, namely
\begin{equation}
  \label{autsk}
  \left\{
    \begin{array}{ll}
      S_k(D^2u)=\lambda (-u)^k &\text{in } \Omega,\\
      u=0 &\text{on } \de\Omega.
    \end{array}
  \right.
\end{equation}

The following existence result holds (see \cite{lionsma} for $k=n$,
and \cite{wangeigen}, \cite{geng} in the general case):
\begin{theo}
  \label{defaut}
  Let $\Omega$ as in \eqref{ipomega}. Then, there exists a positive
  constant $\lambda_k(\Omega)$ depending only on $n,k$, and $\Omega$,
  such that problem (\ref{autsk}) admits a 
  solution $u \in C^2(\Omega)\cap C^{1,1}(\overline{\Omega})$,
  negative in $\Omega$, for
  $\lambda=\lambda_k(\Omega)$ and $u$ is unique up to positive
  scalar multiplication. Moreover, $\lambda_k(\Omega)$ has the following
  variational characterization:
  \begin{equation}
    \label{carvar} \lambda_k(\Omega)=\min_{\substack{u \in
        \Phi_k^2(\Omega) \\
        u \ne 0 }} \displaystyle \frac{\int_{\Omega}(-u)
      S_k(D^2u)\,dx}{\int_{\Omega}(-u)^{k+1}\,dx}.
  \end{equation}
\end{theo}

As matter of fact, if $k<n$ the above theorem holds under a more
general assumption on $\Omega$, namely requiring that $\Omega$ is
strictly $k$-convex (see \cite{wangeigen}, \cite{geng}).

We refer to $\lambda_k(\Omega)$ and $u$, respectively, as the
eigenvalue and eigenfunction of $k$-Hessian 
operator. Moreover, given a function $u\in
\Phi^2_k(\Omega)$, the quantity $\int_{\Omega}(-u) 
S_k(D^2u)\,dx$ is known as $k$-Hessian integral. Using the divergence
form of $\Sk_k$ and the coarea formula, in \cite{trudiso} it is proved 
that
\begin{equation}
\label{carvar2}
  \int_{\Omega}(-u) S_k(D^2u)\,dx = \int_0^{\|u\|_{\infty}} dt \int_{\{u=-t\}}
  H_{k-1}(\{u=-t\}) |Du|^{k} d \mathcal H^{n-1}.
\end{equation}
Hence, the variational formulation \eqref{carvar} can be written in
terms of \eqref{carvar2}.

As matter of fact, we observe that if $k=1$, or $k=n$,
$\lambda_k(\Omega)$ coincides respectively with the first eigenvalue
of the Laplacian operator, or with the eigenvalue of Monge-Amp\`ere
operator. 

If $k=1$, the well-known Faber-Krahn inequality states that
\[
\lambda_1(\Omega) \ge \lambda_1(\Omega\diesis),
\]
where $\Omega\diesis$ is the ball centered at the origin
with the same Lebesgue measure of $\Omega$. Moreover, the equality
holds if $\Omega=\Omega\diesis$. 

In \cite{bt07}, \cite{ga09} it is proved that if $k=n$ and $\Omega$ is
a bounded strictly convex open set, then 
\begin{equation}
  \label{fkdet}
\lambda_n(\Omega) \ge \lambda_n(\Omega_{n-1}^*),  
\end{equation}
where $\Omega^*_{n-1}$ is the ball centered at the origin such that
$W_{n-1}(\Omega)=W_{n-1}(\Omega^*_{n-1})$. We explicitly observe that
if $n=2$, $\Omega^*_{1}$ is the ball with the same perimeter of
$\Omega$. In general, in \cite{ga09} it is proven that if $\Omega$ is
a strictly convex set such that the eigenfunctions have convex level
sets, then, for $2\le k \le n-1$,
\begin{equation}
  \label{fksk}
\lambda_k(\Omega) \ge \lambda_k(\Omega^*_{k-1}),  
\end{equation}
where $\Omega^*_{k-1}$ is the ball centered at the origin such that
$W_{k-1}(\Omega)=W_{k-1}(\Omega^*_{k-1})$.

The additional hypothesis on $\Omega$ seems to be natural. Indeed, for
$k=1$ this is due to the Korevaar concavity maximum principle (see
\cite{kor}), while it is trivial for $k=n$.
For the $k$-Hessian operators, at least in the case $n=3$ and $k=2$,
it in \cite{lmx10} and \cite{sa12} is proved that if $\Omega$ is
sufficiently smooth, the eigenfunctions of $\Sk_2$ have convex level
sets. Up to our knowledge, the general case is an open problem.

We observe that a consequence of the Aleksandrov-Fenchel inequalities
is that, for a set $\Omega$ as in \eqref{ipomega}, then
\begin{equation}
  \label{incl}
  \Omega\diesis=\Omega^*_0 \subseteq \Omega^*_{1} \subseteq \ldots
  \Omega^*_{k-1}\subseteq \Omega^*_{k} \subseteq \ldots \subseteq
  \Omega^*_{n-1}, \quad k=1,\ldots n-1,
\end{equation}
with the equal sign holding if and only if $\Omega$ is a ball.
Indeed, denoted by $R_k$ the radius of $\Omega^*_k$, then by
\eqref{querball} and \eqref{afineq} we have
\[
R_{k-1}=\left(\frac{W_{k-1}(\Omega)}{\omega_n}\right)^{\frac{1}{n-k+1}}
\le
\left(\frac{W_k(\Omega)}{\omega_n}\right)^{\frac{1}{n-k}} \le R_{k}.
\]

From \eqref{incl} and the monotonicity of
$\lambda_k(\cdot)$ with respect to the inclusion of sets, it follows
that, for a set $\Omega$ such that \eqref{fksk} holds, we have
\[
\lambda_k(\Omega) \ge \lambda_k(\Omega^*_{k-1}) \ge
\lambda_k(\Omega^*_k) \ge \ldots\ge \lambda_k(\Omega^*_{n-1}).
\]

\section{Some useful preliminary results}
Let $\Omega$ as in \eqref{ipomega}, and $d(x)$ the distance of a
point $x\in \Omega$ to the boundary $\de \Omega$. We denote by 
\[
\Omega_t=\{ x \in \Omega \colon d(x)>t \}, \quad t\in [0,r_\Omega],
\]
where $r_\Omega$ is the inradius of $\Omega$. The Brunn-Minkowski
inequality for quermassintegrals (\cite[p.339]{schn}) and the
concavity of the distance function give that the function
$W_k(\Omega_t)^{\frac{1}{n-k}}$ is concave in $[0,r_\Omega]$. Hence,
$W_k(\Omega_t)$, $t\in [0,r_\Omega]$ is a decreasing, absolutely
continuous function.

\begin{lemma}
  \label{derquer}
  For any $0\le p \le n-1$, and for almost every $t\in ]0,r_\Omega[$,
  \begin{equation}
    \label{lemmader}
    -\frac{d}{dt} W_p(\Omega_t) \ge (n-p) W_{p+1} (\Omega_t),  
  \end{equation}
  where the equality sign holds if $\Omega$ is a ball.
\end{lemma}
\begin{proof}
  It is not difficult to prove that, if $B_1$ is the unit ball
  centered at the origin, we have
  \[
  \Omega_t + \rho B_1 \subset \Omega_{t-\rho},\quad 0<\rho<t,
  \]
  and the equality holds when $\Omega$ is a ball. Since the
  quermassintegral $W_p(K)$ is monotone with respect to the inclusion
  of convex sets, the above relation and \eqref{steinercons} give that
  \[
 -\frac{d}{dt} W_p(\Omega_t) = \lim_{\rho\rightarrow 0^+}
 \frac{W_p(\Omega_{t-\rho}) - W_p(\Omega_t)}{\rho} \ge \lim_{\rho\rightarrow 0^+}
 \frac{W_p(\Omega_{t}+\rho B_1) - W_p(\Omega_t)}{\rho} = (n-p)W_{p+1}(\Omega_t)
 \]
 for almost every  $t\in]0,r_\Omega[$.
\end{proof}
\begin{rem}
  As matter of fact, it is well-known that the inequality
  \eqref{lemmader} holds as an equality when $p=0$. In such a
  case $W_0(\Omega_t)=|\Omega_t|$,
  $W_1(\Omega_t)=nP(\Omega_t)$. Moreover, using the coarea formula,
  and being $d\in W^{1,\infty}(\Omega)$ with $|Dd|=1$ a.e., we have for
  a.e. $t\in ]0,r_\Omega[$
  \begin{equation}
    \label{dermis}
  -\frac{d}{dt} |\Omega_t|= \int_{\{d=t\}} \frac{1}{|Dd|} d\mathcal
  H^{n-1} = P(\Omega_t).
  \end{equation}
\end{rem}
An immediate consequence of Lemma \ref{derquer} is the following result.
\begin{lemma} \label{derquer2}
  Let $u(x)=f(d(x))$, where $f\colon[0,+\infty[\rightarrow[0,+\infty[$
  is a strictly $C^1$ function with $f(0)=0$. Set
  \[
  E_t=\{x\in\Omega\colon u(x)>t\}=\Omega_{f^{-1}(t)}.
  \]
  Then, for $0\le p \le n-1$, and for a.e. $t\in]0,r_\Omega[$, 
  \[
  -\frac{d}{dt} W_{p}(E_t) \ge (n-p) \frac{W_{p+1} (E_t)}{|Du|_{u=t}}.
  \]
\end{lemma}
We conclude the Section with other two results which will be used in
next sections. The first one concerns an integral inequality, while
the second gives an approximation of the distance 
with suitable smooth functions.
\begin{lemma}
  \label{lemd}
  Let $f\colon [0,+\infty[\rightarrow \R$ a $C^1$ nondecreasing
  function. Denoted with $P(t)=P(\Omega_t)$ and with $r_\Omega$ the
  inradius of $\Omega$, then
  \begin{equation}
    \label{eq:1}
    \int_0^{r_\Omega} f(t)P(t)\, dt \ge P(\Omega)
    \int_0^{|\Omega|/P(\Omega)} f(t)\, dt.
  \end{equation}
\end{lemma}
\begin{proof}
  We first observe that from \eqref{dermis} it holds that
  \[
  |\Omega| = \int_0^{r_\Omega} P(t)dt \le P(\Omega) r_\Omega.
  \]
  Then we can define the auxiliary function
  \[
  \tilde P(t)=
  \begin{cases}
    P(\Omega) & \text{if } 0\le t \le \frac{|\Omega|}{P(\Omega)},\\
    0 & \text{if } \frac{|\Omega|}{P(\Omega)}< t \le r_\Omega.
  \end{cases}
  \]
  It is easy to see that
  \begin{equation}
    \label{eq:2}
    \int_0^s P(t) dt \le   \int_0^s \tilde P(t) dt,\quad \forall
  s\in[0,r_\Omega[,
\end{equation}
and
\begin{equation}
  \label{eq:3}
  \int_0^{r_\Omega} P(t) dt =
  \int_0^{r_\Omega} \tilde P(t) dt.
\end{equation}
  Indeed, \eqref{eq:2} is obvious if $s< |\Omega|/P(\Omega)$. Otherwise,
  \[
  \int_0^s P(t)\, dt = |\Omega| - |\Omega_s| \le
  \int_0^{{|\Omega|}/{P(\Omega)}} P(\Omega)\, dt = \int_0^s \tilde P(t)\,dt,
  \]
  and we also have \eqref{eq:3}, since $|\Omega_{r_\Omega}|=0$. Then,
  being $f$ increasing, an integration by parts shows that
  \eqref{eq:2} and \eqref{eq:3} imply \eqref{eq:1}.
\end{proof}
\begin{prop}\label{approx}
  Suppose that $\Omega$ verifies \eqref{ipomega}. Then, there exists a
  sequence of functions $\{d_\eps(x)\}_{\eps>0}$, $x\in
  \bar{\Omega}$, such
  that:
  \begin{enumerate}
  \item $d_\eps$ concave in $\Omega$, $d_\eps\equiv 0$ on $\de \Omega$ and
    $d_\eps\in C^2(\Omega)\cap C(\bar\Omega)$;
  \item $0\le d_\eps \le d$, and $d_\eps\rightarrow d$ uniformly in
    $\bar \Omega$;
  \item $|Dd_\eps|\le 1$ in $\Omega$.
  \end{enumerate}
\end{prop}
\begin{proof}
  As well-known, the function $d$ is the unique viscosity solution of
  the Dirichlet problem 
  \[
  \left\{
    \begin{array}{ll}
      |D w|^2=1 &\text{in } \Omega,\\
      w=0 &\text{on }\de\Omega.
    \end{array}
  \right.
  \]
  
  Using the standard vanishing viscosity argument, the required
  sequence can be obtained by solving the problems
  \begin{equation}
    \label{apprd}
  \left\{
    \begin{array}{ll}
      \eps \Delta w - |D w|^2+1=0 &\text{in } \Omega,\\
      w=0 &\text{on }\de\Omega.
    \end{array}
  \right.
\end{equation}
The existence and uniqueness of a solution in $C^2(\Omega)\cap
C(\bar\Omega)$ of \eqref{apprd} can be proved by making the change of
variable 
\[
z=\exp\left\{-\frac{w}{\eps} \right\} -1.
\]
Then, $w$ is a solution to \eqref{apprd} if and only if $z\in C^2(\Omega)\cap
  C(\bar\Omega)$ verifies 
 \begin{equation}
    \label{apprd2}
  \left\{
    \begin{array}{ll}
      \eps^2 \Delta z -z= 1 &\text{in } \Omega,\\
      w=0 &\text{on }\de\Omega.
    \end{array}
  \right.
\end{equation}
It is well-known that problem \eqref{apprd2} admits a unique solution
$C^2(\Omega)\cap C(\bar\Omega)$. Hence, the function $w=-\eps
\log(z+1)$ is the unique solution of \eqref{apprd}. For any $\eps>0$,
we choose $d_\eps=w$.

By comparison arguments, it is possible to show that $d_\eps$
satisfies (2) and (3) (see for instance \cite{schie} for the
details). Finally, the concavity of $d_\eps$ follows applying the
Korevaar concavity maximum principle to
\eqref{apprd} (see \cite{kor}).
\end{proof}

\section{Main results}
In this section we state and prove the main results on upper bounds
for the eigenvalue of $\Sk_k$. For ease of reading, we organize the
Section in three different subsections. 

The first aim is to prove an upper bound of $\lambda_k(\Omega)$ by
means of a suitable isoperimetric deficit. To get such estimate we
have to study separately the case $k=n$ and the 
case $1\le k \le n-1$. 
We start by recalling some properties of the eigenfunctions of $\Sk_k$
in a ball.

Let $B_R$ be a ball of $\R^n$ centered at the origin with radius $R$,
and $v\in C^2(B_R)\cap C(\bar B_R)$ be an eigenfunction of the
$k$-Hessian operator in $B_R$.
This means that $v$ verifies
\begin{equation}
  \label{autraddet}
  \left\{
    \begin{array}{ll}
      \Sk_k(D^2v)=\lambda_k(B_R) (-v)^k &\text{in } B_R,\\
      v=0 &\text{on } \de B_R.
    \end{array}
  \right.
\end{equation}
It is known that $v$ is a negative, convex, radially increasing
smooth function. We have that:
\begin{equation}\left\{
    \begin{array}{ll}
      \label{vrad}
v(x)=\varphi(r), & r=|x|,\; x\in B_R,\\
\varphi<0\text{ in }[0,R[,&\varphi(R)=0,\\
\varphi'>0\text{ in }]0,R],&\varphi'(0)=0
\end{array}
\right.
\end{equation}
(see \cite{bt07}, \cite{ga09}).

\subsection{Stability estimates: the case of Monge-Amp\`ere operator}
Let us consider problem \eqref{autraddet} with $k=n$ and
$B_R=\Omega_{n-1}^*$, where  $\Omega_{n-1}^*$ is the ball centered at
the origin and radius $R$ such that
$W_{n-1}(\Omega)=W_{n-1}(\Omega_{n-1}^*)$. Here, $v$ denotes an
eigenfunction relative to $\lambda_n(\Omega_{n-1}^*)$.
Recall that the Faber-Krahn inequality \eqref{fkdet} holds. 

Together with \eqref{fkdet}, the following result gives a quantitative
estimate of $\lambda_n(\Omega)$.
\begin{theo}
  \label{ma}
  Let $\Omega$ be as in \eqref{ipomega}. Then
   \begin{equation}
     \label{quadet}
    \frac{\lambda_n(\Omega) -
      \lambda_n(\Omega^*_{n-1})}{\lambda_n(\Omega)} \le C_\Omega
    \big[|\Omega^*_{n-1}|-|\Omega|\big],
  \end{equation}
  where $C_\Omega=\left(\frac{\| v \|_{\infty}}{\|v\|_{n+1}}\right)^{n+1}$.
\end{theo}
\begin{proof}
  Without loss of generality,
 we can suppose that the quantity in the right-hand side of
 \eqref{quadet} is smaller than $1$. Otherwise, \eqref{quadet} is
 trivial.  Let $R>0$ such that $B_R=\Omega^*_{n-1}$, and define 
  \[
  u_\eps(x)= \varphi(R-d_\eps(x)),\quad x\in\Omega,
  \]
  where $\varphi$ is given in \eqref{vrad} and $d_\eps$ is
  the approximation of the distance function to the boundary of
  $\Omega$ given in Proposition \ref{approx}. For any $\eps>0$, the
  function $u_\eps$ is well defined, being $d_\eps
  \le d \le r_\Omega$, and $r_\Omega\le R$. Last inequality is true since, by
  the definition of $W_{n-1}$ and using the Aleksandrov-Fenchel
  inequality for $j=n-1$ and $i=0$, we have 
  \[
  \omega_nR = W_{n-1}(\Omega^*_{n-1}) =W_{n-1}(\Omega) \ge
  \omega_n^{1-\frac 1 n} |\Omega|^{\frac 1 n} \ge \omega_n r_\Omega. 
  \]
As matter of fact, denoting the function $g(t)=|Dv|_{\{v=-t\}}$, $0\le
t \le \|v\|_{\infty}$, by construction, the function $u_\eps$ has the
following properties: 
\begin{equation}\left\{
    \begin{array}{l}
      \label{ueps}
u_\eps(x)\in \Phi_n^2(\Omega),\\
|Du_\eps|_{\{u_\eps=-t\}} \le g(t), \\
\|u_\eps\|_{\infty}\le \|v\|_{\infty},\\
u_\eps(x)\rightarrow u(x)=\varphi(R-d(x))\text{ uniformly in }\bar\Omega.
\end{array}
\right.
\end{equation}
Let us define
\[
  E_t=\{x\in\Omega\colon u<-t\},\quad B_t=\{ x\in \Omega^*_{n-1}
  \colon v<-t\}.
  \]
  $E_t$ is a convex set, while $B_t$ is a ball centered at the
  origin. 

  Lemma \ref{derquer2} implies that
  \[
  -\frac {d}{dt} W_{n-1}(E_t) \ge \frac{\omega_n}{g(t)}=  -\frac
  {d}{dt} W_{n-1}(B_t).
  \]
  Together with the initial condition $W_{n-1}(E_0)=W_{n-1}(B_0)$, we
  have that
  \[
  W_{n-1}(E_t) \le W_{n-1} (B_t), \quad 0<t<\sup(-v).
  \]
  Applying the Aleksandrov - Fenchel inequalities, the above
  inequality gives that
  \begin{equation}\label{perin2}
  P(E_t) \le P(B_t).
  \end{equation}
   Now, denote with $\mu(t)=|E_t|$ and $\nu(t)=|B_t|$. Using the coarea
 formula and the inequality \eqref{perin2}, we have that
  \begin{equation*}
  -\mu'(t)=\int_{\{u=-t\}} \frac{1}{|Du|}d\mathcal H^{n-1} =
  \frac{P(E_t)}{g(t)} \le \frac{P(B_t)}{g(t)}
  =\int_{\{v=-t\}} \frac{1}{|Dv|}d\mathcal H^{n-1} =-\nu'(t),
\end{equation*}
and then $\nu-\mu$ is a decreasing function. Hence,
\begin{multline}
  \label{stimedendet}
\int_{\Omega} (-u)^{n+1} dx =\int_0^{\|u\|_{\infty}} (n+1) t^n \mu(t)
dt =\\ = \int_0^{\|v\|_{\infty}} (n+1) t^n \nu(t) dt
-\int_0^{\|v\|_{\infty}} (n+1) t^n \left[\nu(t) - \mu(t) \right] dt
\ge \\ \ge \int_{\Omega^*_{n-1}} v^p dx - \left(
  |\Omega^*_{n-1}|-|\Omega|\right)\| v \|_{\infty}^{n+1}.
\end{multline}
Then, from the uniform convergence
of $u_\eps$ to $u$ and \eqref{stimedendet} we get that
\begin{equation}
  \label{denom}
  \lim_{\eps\rightarrow 0^+} \int_\Omega (-u_\eps)^{n+1} dx = \int_\Omega (-u)^{n+1} dx
\ge\int_{\Omega_{n-1}^*} (-v)^{n+1} dx - \left(
  |\Omega^*_{n-1}|-|\Omega|\right)\| v \|_{\infty}^{n+1}.
\end{equation}
On the other hand, \eqref{carvar2} and \eqref{ueps} imply that
  \begin{multline*}
  \int_\Omega (-u_\eps)\det(D^2 u_\eps) dx =
  \int_0^{\|u_\eps\|_\infty} dt\int_{\{u_\eps=-t\}} H_{n-1}
  (\{u_\eps=-t\}) |D u_\eps|^n d\mathcal H^{n-1} \le \\ \le
  \int_0^{\|u_\eps\|_\infty} g(t) dt  \int_{\{u_\eps=-t\}} H_{n-1}
  (\{u_\eps=-t\}) d\mathcal H^{n-1} =\\ =  n \omega_n \int_0^{\|u_\eps\|_\infty}
  g(t) dt  \le  n \omega_n \int_0^{\|v\|_\infty}
  g(t) dt =\\ =\int_0^{\|v\|_\infty} dt\int_{\{v=-t\}} H_{n-1}
  (\{v=-t\}) |D v|^n d\mathcal H^{n-1} =\int_{\Omega^*_{n-1}} (-v)\det(D^2 v) dx.
\end{multline*}

Finally, putting together \eqref{denom} and the above inequality, we
have that
\begin{multline}
\lambda_n(\Omega) \le \liminf_{\eps\rightarrow 0^*} \frac{  \int_\Omega
  (-u_\eps)\det(D^2 u_\eps) dx }{\int_\Omega (-u_\eps)^{n+1} dx}
\le \frac{\int_{\Omega^*_{n-1}} (-v)\det(D^2 v)
  dx}{\int_{\Omega^*_{n-1}} (-v)^{n+1}dx -\left(
  |\Omega^*_{n-1}|-|\Omega|\right)\| v \|_{\infty}^{n+1}} =
\\ = \frac{\lambda_n(\Omega^*_{n-1})}{ 1 -
  (|\Omega^*_{n-1}|-|\Omega|)\left(\frac{\| v
      \|_{\infty}}{\|v\|_{n+1}}\right)^{n+1}},
\end{multline}
and then 
\[
 \frac{\lambda_n(\Omega) - \lambda_n(\Omega^*_{n-1})}{\lambda_n(\Omega)} \le
\left(|\Omega^*_{n-1}|-{|\Omega|} \right) \left(\frac{\| v
    \|_{\infty}}{\|v\|_{n+1}}\right)^{n+1}.
\]
\end{proof}
\begin{rem}
  If $n=2$, the estimate \eqref{quadet} becomes
  \[
  \frac{\lambda_2(\Omega) -
    \lambda_2(\Omega^\star)}{\lambda_2(\Omega)} \le 
  \frac{C_\Omega}{4\pi}\left(P^2(\Omega)- 
   4\pi |\Omega| \right), 
  \]
  where $\Omega^\star=\Omega^*_1$ is the ball with the same perimeter
  than $\Omega$.
\end{rem}
\subsection{Stability estimates: the case of $k$-Hessian operator,
  $k<n$}
Now we consider problem \eqref{autraddet} with $1\le k\le n-1$ and
$B_R=\Omega_{k}^*$, where  $\Omega_{k}^*$ is the ball centered at
the origin and radius $R$ such that
$W_{k}(\Omega)=W_{k}(\Omega_{k}^*)$. As before, $v$ denotes an
eigenfunction relative to $\lambda_k(\Omega_k^*)$.
\begin{theo}
  \label{estsk}
  Let $\Omega$ be as in \eqref{ipomega}, and $1\le k\le n-1$. Then
  \begin{equation}
    \label{tesi}
    \frac{\lambda_k(\Omega) -
      \lambda_k(\Omega^*_{k})}{\lambda_k(\Omega)} \le C_\Omega
    \left(|\Omega^*_k|-|\Omega| \right),
  \end{equation}
  where $C_\Omega= \left(\frac{\| v
      \|_{\infty}}{\|v\|_{k+1}}\right)^{k+1}$.
\end{theo}
\begin{proof}
  We follow the lines of the proof of Theorem \ref{ma}. First, suppose
  that the quantity in the right-hand side of \eqref{tesi} is smaller
  than 1. Let $R>0$ be such that $B_R=\Omega^*_{k}$, and 
  \[
  u_\eps(x)=\varphi(R-d_\eps(x)),\quad x\in \Omega.
  \]
 The function $u_\eps$ is well defined, since by Aleksandrov-Fenchel
 inequalities we have
 \[
   \omega_nR^{n-k} = W_{k}(\Omega^*_{k}) =W_{k}(\Omega) \ge
  \omega_n^{1-\frac {n-k}{n}} |\Omega|^{\frac {n-k}{n}} \ge \omega_n
  r_\Omega^{n-k}. 
  \]
  By construction, $u_\eps$ has the
  following properties:
  \begin{equation*}\left\{
      \begin{array}{l}
        u_\eps(x)\in \Phi_k^2(\Omega),\\
        |Du_\eps|_{\{u_\eps=-t\}} \le g(t):=|Dv|_{v=-t}, \\
        \|u_\eps\|_{\infty}\le \|v\|_{\infty},\\
        u_\eps(x)\rightarrow u(x)=\varphi(R-d(x))\text{ uniformly in
        }\bar\Omega. 
      \end{array}
      \right.
  \end{equation*}
  For $t\ge 0$, we set
  \[
  E_t=\{x\in\Omega\colon -u> t\},\quad B_t=\{ x\in \Omega^*_{k}
  \colon -v>t\}.
  \]
  $E_t$ is a convex set, while $B_t$ is a ball centered at the origin.

  By Lemma \ref{derquer2} and the Aleksandrov-Fenchel inequalities, if
  $1\le k <n-1$,
  we have
  \begin{align*}
    -\frac {d}{dt} W_{k}(E_t) &\ge (n-k) \frac{W_{k+1} (E_t)}{g(t)} \\
    &\ge
  (n-k)\omega_n^{\frac{1}{n-k-1}} \frac{W_{k}(E_t)^{\frac{n-k-1}{n-k}}}{g(t)},
  \end{align*}
  and
  \begin{align*}
    -\frac {d}{dt} W_{k}(B_t) &= (n-k) \frac{W_{k+1} (B_t)}{g(t)}\\
    &=(n-k+1)\omega_n^{\frac{1}{n-k}}  \frac{W_{k}(B_t)^{\frac{n-k}{n-k}}}{g(t)}.
  \end{align*}

  If $k=n-1$, being $W_n(K)=\omega_n$, we write simply that
  \[
  -\frac {d}{dt} W_{n-1}(E_t) \ge \frac{\omega_n}{g(t)},
  \]
  and
  \[
  -\frac {d}{dt} W_{n-1}(B_t) = \frac{\omega_n}{g(t)}.
  \]

  Being $W_{k}(E_0)=W_{k}(B_0)$, by the classical comparison
  theorems for differential inequalities, we get that
  \begin{equation}\label{wkin}
    W_{k}(E_t) \le W_{k}(B_t),\quad 0< t < \|v\|_{\infty}.
  \end{equation}
  The inequality \eqref{wkin} implies that
  \begin{equation*}
    P(E_t) \le P(B_t).
  \end{equation*}
  Indeed, this is trivial if $k=1$. In the case $2\le k \le
  n-1$, this follows using the Aleksandrov-Fenchel inequalities
  \eqref{af-2} in \eqref{wkin}, and recalling that \eqref{af-2} holds
  as an equality for the sets $B_t$.

 Now, reasoning similarly as in the proof of Theorem \eqref{ma}, 
it follows that
\[
\lim_{\eps\rightarrow 0^+} \int_\Omega (-u_\eps)^{k+1} dx =
\int_\Omega (-u)^{k+1} dx \ge\int_{\Omega_k^*} (-v)^{k+1} dx -\left(
  |\Omega^*_{k}|-|\Omega|\right)\| v \|_{\infty}^{k}.
\]
Moreover, recalling the properties of $u_\eps$ and observing that the
level set $E_k^\eps= \{u_\eps < -t \}$ are contained in
$E_t=\{u<-t\}$, by \eqref{carvar2} we get that
\begin{multline}
  \label{gap}
  \int_\Omega (-u_\eps)\Sk_k(D^2 u_\eps) dx =
  \int_0^{\|u_\eps\|_\infty} dt\int_{\{u_\eps=-t\}} H_{k-1}
  (\{u_\eps=-t\}) |D u_\eps|^k d\mathcal H^{n-1} \le \\ \le
  \int_0^{\|v\|_\infty} g(t) dt  \int_{\{u_\eps=-t\}} H_{k-1}
  (\{u_\eps=-t\}) d\mathcal H^{n-1} =\\ =  n \int_0^{\|v\|_\infty}
  g(t) W_k(E^\eps_t) dt  \le  n \int_0^{\|v\|_\infty}
  g(t) W_k(E_t) dt \le  n \int_0^{\|u\|_\infty}
  g(t) W_k(B_t) dt =\\ =\int_0^{\|v\|_\infty} dt\int_{\{v=-t\}} H_{k-1}
  (\{v=-t\}) |D v|^k d\mathcal H^{n-1} =\int_{\Omega^*_k} (-v)\Sk_k(D^2 v) dx.
\end{multline}

Finally,
\begin{multline*}
\lambda_k(\Omega) \le \liminf_{\eps\rightarrow 0^+} \frac{  \int_\Omega
  (-u_\eps)\Sk_k(D^2 u_\eps) dx }{\int_\Omega (-u_\eps)^{k+1} dx}
\le \frac{\int_{\Omega^*_{k}} (-v)\Sk_k(D^2 v)
  dx}{\int_{\Omega^*_{k}} (-v)^{k+1}dx -\left(
  |\Omega^*_{k}|-|\Omega|\right)\| v \|_{\infty}^{k}} =
\\ = \frac{\lambda_k(\Omega^*_{k})}{ 1 -
  (|\Omega^*_{k}|-|\Omega|)\left(\frac{\| v
      \|_{\infty}}{\|v\|_{k+1}}\right)^{k+1}},
\end{multline*}
and we can conclude that
\begin{equation*}
\frac{\lambda_k(\Omega) - \lambda_k(\Omega^*_{k})}{\lambda_k(\Omega)} \le
\left(|\Omega^*_{k}|-|\Omega|\right) \left(\frac{\| v
    \|_{\infty}}{\|v\|_{k+1}}\right)^{k+1}.
\end{equation*}
\end{proof}
\begin{rem}
  We observe that if we choose $\Omega$ in the class of sets such that
  the Faber-Krahn inequality 
  \[
  \lambda_k(\Omega)\ge \lambda_k(\Omega^*_{k-1}) 
  \]
  holds (see Section 1.2), then \eqref{estsk} gives a quantitative
  estimate for $\lambda_k$ in terms of an isoperimetric
  deficit. Indeed, in such a case, being $\lambda_k(\cdot)$ decreasing
  with respect to the inclusion of sets, we have
  \[
   0\le \frac{\lambda_k(\Omega) -
     \lambda_k(\Omega^*_{k-1})}{\lambda(\Omega)}\le
   \frac{\lambda_k(\Omega) - 
    \lambda_k(\Omega^*_{k})}{\lambda_k(\Omega)} \le C_\Omega
  \left(|\Omega^*_k|- |\Omega| \right).
   \]
 \end{rem}

In the last subsection we give an estimate of $\lambda_k(\Omega)$ that
generalizes the one obtained by Makai in \cite{mak} for the first
eigenvalue of the Laplacian.
\subsection{An upper bound for the eigenvalue of $\Sk_k$, $1\le k \le
  n$}
\begin{theo}
Let $\Omega$ verifies \eqref{ipomega}, and let $\lambda_k(\Omega)$ be
the eigenvalue of the $k$-Hessian operator $\Sk_k$ in $\Omega$, with $1\le
k \le n$. Then the following upper bound for $\lambda_k(\Omega)$
holds:  
\begin{equation}
  \label{polyaut}
  \lambda_k(\Omega) \le \frac{n(k+2)}{n-k +1}
  \frac{P(\Omega)^{k+1}}{|\Omega|^{k+2}} W_{k-1}(\Omega).
\end{equation}
\end{theo}

\begin{proof}
  Let $d_\eps$ the sequence given in Proposition \ref{approx}. Recall
  that $r_\Omega$ is the inradius of $\Omega$.
  By \eqref{carvar2} and Lemma \ref{derquer} we have that
  \begin{multline}\label{contodist1}
    \int_\Omega (d_\eps)\Sk_k(D^2 d_\eps) dx =
    \int_0^{\|d_\eps\|_\infty} dt\int_{\{d_\eps=t\}} H_{k-1}
    (\{d_\eps=t\}) |D u_\eps|^k d\mathcal H^{n-1} \le \\ \le
    \int_0^{r_\Omega}dt  \int_{\{d_\eps=t\}} H_{k-1}
    (\{d_\eps=t\}) d\mathcal H^{n-1} = n\int_0^{r_\Omega}
   W_{k}(\{d_\eps>t\}) dt \le \\ \le n\int_0^{r_\Omega}
   W_{k}(\Omega_t) dt \le \frac{n}{n-k+1} \int_0^{r_\Omega}
   -\frac{d}{dt} W_{k-1}(\Omega_t) dt= \frac{n}{n-k+1} W_{k-1}(\Omega),
 \end{multline}
 while, using the coarea formula and Lemma \ref{lemd} with
 $f(t)=t^{k+1}$, we have that
 \begin{equation}
   \label{contodist2}
   \int_\Omega d^{k+1} dx = \int_0^{r_\Omega} t^{k+1} P(\Omega_t) dt \ge
   P(\Omega) \int_0^{|\Omega|/P(\Omega)} t^{k+1} dt = \frac{1}{(k+2)}
   \frac{|\Omega|^{k+2}}{P(\Omega)^{k+1}}.
 \end{equation}
 Hence, recalling also that $d_\eps\rightarrow d$ uniformly in $\bar
 \Omega$, by \eqref{carvar}, \eqref{contodist1} and \eqref{contodist2} we get
 \[
 \lambda_k(\Omega) \le \liminf_{\eps\rightarrow 0^+} \frac{  \int_\Omega
   d_\eps \Sk_k(D^2 d_\eps) dx }{\int_\Omega d_\eps^{k+1} dx} \le
 \frac{n(k+2)}{n-k +1} \frac{P(\Omega)^{k+1}}{|\Omega|^{k+2}}
 W_{k-1}(\Omega).
 \]
\end{proof}
\begin{rem}
  We emphasize two particular cases of \eqref{polyaut}.
  First, for $k=1$ it becomes
  \[
  \lambda_1(\Omega) \le 3 \frac{P^2(\Omega)}{|\Omega|^2},
  \]
  that is exactly the Makai estimate contained in
  \cite{mak}. Moreover, for $k=n=2$, the estimate \eqref{polyaut} is
  \[
  \lambda_2(\Omega) \le 4 \frac{P(\Omega)^4}{|\Omega|^4}. 
  \]
\end{rem}
\begin{rem}
  \label{cri}
  We recall that an upper bound of the Dirichlet
  eigenvalue of the Monge-Amp\`ere operator on convex smooth domain with
  fixed measure has been given in \cite{bntpoin}. More precisely,
  the authors prove that
  \begin{equation}
    \label{estellips}
    \lambda_n(\Omega) \le \lambda_n(\Omega\diesis).
  \end{equation}
 Furthermore, according to the invariance under volume
 preserving affine transformations of $\det D^2 u$, they prove that
 the equality holds if and only if $\Omega$ is an ellipsoid. Clearly,
 if $\Omega$ is a 
 smooth convex set with fixed measure, the quantities in
 \eqref{estellips} remain bounded, while the right-hand side of
 \eqref{polyaut} diverges if, for example, $P(\Omega)\rightarrow
 +\infty$. As matter of fact, \eqref{estellips} cannot hold for
 $\lambda_k(\Omega)$, since it may diverge as $P(\Omega)\rightarrow
 +\infty$ and $|\Omega|$ fixed, as shown by the following example.
\end{rem}
\begin{ex}
  \label{exellipse}
  
  \begin{figure}[h]
    \includegraphics[scale=.5]{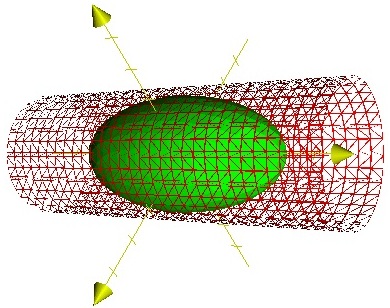}
    \centering \caption{The domains $\mathcal E_a$ and $\mathcal C_{2a}$
      of Example \ref{exellipse} }
  \end{figure}
  For sake of simplicity, we consider the case $n=3$ and $k=2$.

  Let $\mathcal E_a$ be the ellipsoid $\mathcal E_a=\{(x,y,z)\in
  \mathbb R^3\colon \frac{1}{a^2}(x^2+y^2)+a^4 z^2=1\}$. Clearly,
  $|\mathcal E_a|=\frac 4 3 \pi$,  and $\mathcal E_a \cap \{z=0\}=D_a$ is the
  disk of $\R^2$ centered at the origin with radius $a>0$. 
  Let $u$ be an eigenfunction of $\Sk_2$ in $\mathcal E_a\subset
  \R^3$, relative to $\lambda$, and $w$ be an eigenfunction of the
  Monge-Amp\`ere operator in $D_{2a}\subset \R^2$ relative to
  $\mu$, that is 
\begin{equation*}
  \left\{
    \begin{array}{ll}
      \Sk_2(D^2u)=\lambda (-u)^2 &\text{in } \mathcal E_a,\\
      u=0 &\text{on } \de \mathcal E_a,
    \end{array}
  \right. \quad\qquad
  \left\{
    \begin{array}{ll}
      \det (D^2w)=\mu (-w)^2 &\text{in } D_{2a},\\
      w=0 &\text{on } \de D_{2a}.
    \end{array}
  \right.
\end{equation*}
By the definition of $\Sk_2$, the function
\[
v(x,y,z):=w(x,y),\;\;
(x,y,z)\in\mathcal C_{2a}= \{(x,y,z) \colon (x,y)\in
D_{2a},\, z\in \R\}
\]
verifies 
\begin{equation*}
  \left\{
    \begin{array}{ll}
      \Sk_2(D^2v)=\mu (-v)^2 &\text{in } \mathcal E_{a},\\
      v < 0 &\text{on } \de \mathcal E_{a}.
    \end{array}
  \right.
\end{equation*}
Then, an argument based on the maximum principle for fully
nonlinear elliptic equations (see \cite[Theorem 17.1]{gt}) gives that
\begin{equation}
  \label{comp}
  \lambda \ge \mu.
\end{equation}
Finally, being
\[
\mu = \frac{\int_{D_{2a}}(-w)\det D^2 w
  dx}{\int_{D_{2a}}(-w)^3 dx} \sim \frac {1}{|D_{2a}|^2}, 
\]
by \eqref{comp} we have that $\lambda\rightarrow +\infty$ as
$a\rightarrow 0$. 
\end{ex}

\end{document}